\documentclass[12pt]{amsart}

\usepackage{amscd, amsfonts, amsmath, amssymb, amsthm, mathrsfs, appendix, enumerate, graphicx, flexisym, epsfig, float, graphicx, hyperref, pdfpages, tikz, mathtools, comment, tikz-cd}
\usepackage[all,cmtip]{xy}
\usepackage{pinlabel}

\makeatletter
\@namedef{subjclassname@2020}{Mathematics Subject Classification \textup{2020}}

\newsavebox{\@brx}
\newcommand{\llangle}[1][]{\savebox{\@brx}{\(\m@th{#1\langle}\)}%
  \mathopen{\copy\@brx\kern-0.5\wd\@brx\usebox{\@brx}}}
\newcommand{\rrangle}[1][]{\savebox{\@brx}{\(\m@th{#1\rangle}\)}%
  \mathclose{\copy\@brx\kern-0.5\wd\@brx\usebox{\@brx}}}

\newtheorem{theorem}{Theorem}[section]

\newtheorem{lemma}[theorem]{Lemma}
\newtheorem{proposition}[theorem]{Proposition}
\theoremstyle{definition}

\newtheorem{definition}[theorem]{Definition}
\newtheorem{example}[theorem]{Example}
\newtheorem{remark}[theorem]{Remark}

\numberwithin{equation}{subsection}
\newtheorem*{ack}{Acknowledgement}
\usepackage[all,cmtip]{xy}

\newcommand{\Aut}{\operatorname{Aut}}

\newcommand{\Stab}{\operatorname{Stab}}

\newcommand{\Inn}{\operatorname{Inn}}
\newcommand{\Out}{\operatorname{Out}}

\newcommand{\Homeo}{\operatorname{Homeo}}

\newcommand{\Mod}{\mathrm{MCG}}

\newcommand{\m}{\mathfrak{t}}
\newcommand{\p}{\hat{\mathfrak{t}}}
\begin{document}

\title{Twisted Rokhlin property for mapping class groups}

\author{Pravin Kumar}
\email{pravin444enaj@gmail.com}
\address{Department of Mathematical Sciences, Indian Institute of Science Education and Research (IISER) Mohali, Sector 81,  S. A. S. Nagar, P. O. Manauli, Punjab 140306, India.}

\author{Apeksha Sanghi}
\email{apekshasanghi93@gmail.com}

\address{Department of Mathematical Sciences, Indian Institute of Science Education and Research (IISER) Mohali, Sector 81,  S. A. S. Nagar, P. O. Manauli, Punjab 140306, India.}
\author{Mahender Singh}
\email{mahender@iisermohali.ac.in}
\address{Department of Mathematical Sciences, Indian Institute of Science Education and Research (IISER) Mohali, Sector 81,  S. A. S. Nagar, P. O. Manauli, Punjab 140306, India.}



\begin{abstract}
In this paper, generalising the idea of the Rokhlin property, we explore the concept of the twisted Rokhlin property of topological groups. A topological group is said to exhibit the twisted Rokhlin property if, for each automorphism $\phi$ of the group, there exists a $\phi$-twisted conjugacy class that is dense in the group. We provide a complete classification of connected orientable infinite-type surfaces without boundaries whose mapping class groups possess the twisted Rokhlin property. Additionally, we prove that the mapping class groups of the remaining surfaces do not admit any dense $\phi$-twisted conjugacy class for any automorphism $\phi$. This supplements the recent work of Lanier and Vlamis on the Rokhlin property of big mapping class groups. We also prove that the mapping class group of each connected orientable infinite-type surface without boundary possesses the $R_\infty$-property.

\end{abstract}

\maketitle
% \tableofcontents

\section{Introduction}
A topological group is said to exhibit the Rokhlin property if it contains a dense conjugacy class. From the point of view of dynamics, the Rokhlin property is equivalent to the existence of a dense orbit under the conjugation action of the topological group on itself. This perspective allows us to consider more general actions of the topological group arising from its structure.
\par

Given an automorphism $\phi$ of a group $G$, the map $(g,h) \mapsto gh\phi(g)^{-1}$, defines a left action of $G$ on itself. Two elements $x, y \in G$ are said to be $\phi$-twisted conjugate if they are in the same orbit under this action. The resulting orbits  are referred as {\it $\phi$-twisted conjugacy classes}.  The idea of twisted  conjugacy  arose from the work of Reidemeister \cite{MR1513064}. It has deep connection with Nielsen fixed-point theory \cite{MR1697460}, and appears in  Arthur--Selberg  theory \cite{MR1176101}, algebraic geometry \cite{grothendieck1977}, and representation theory \cite{MR0000998}, to name a few. In recent years, there has been a tremendous amount of work on understanding twisted conjugacy in various classes of groups. See, for example \cite{MR4525669, MR2644279}, for recent works on twisted conjugacy classes in mapping class groups and braid groups.
\par 

In this paper, we combine the concept of the Rokhlin property with that of twisted conjugacy.  Let $G$ be a topological group and $\phi$ an automorphism of the underlying abstract group. We say that $G$ has the {\it $\phi$-twisted Rokhlin property} if $G$ admits a dense $\phi$-twisted conjugacy class. Further, $G$ is said to have the {\it twisted Rokhlin property} if it has the $\phi$-twisted Rokhlin property for each automorphism $\phi$ of the underlying abstract group. Clearly, if a topological group has the twisted Rokhlin property, then it has the Rokhlin property.
\par 

The mapping class group of a connected orientable surface is the group of isotopy classes of orientation-preserving self-homeomorphisms of the surface. These groups inherit a non-trivial quotient topology from the compact-open topology on the group of orientation-preserving self-homeomorphisms of the surface. The aim of this paper is to classify  connected orientable surfaces without boundaries whose mapping class groups have the twisted Rokhlin property. Our approach is built upon the recent work \cite{MR4492497} of Lanier and Vlamis, where they classify connected orientable surfaces without boundaries whose mapping class groups admit  the Rokhlin property (see also \cite{MR4477213} for an independent such classification).
\par

The paper is organised as follows. In Section \ref{section 2}, we recall basic results on infinite type surfaces, mapping class groups and curve graphs, that we shall need in the latter sections. In Section \ref{section 3}, we introduce the twisted Rokhlin property of topological groups and make some basic observations. We prove that  if $\phi, \psi$ are two automorphisms of a topological group $G$ such that $\psi \in \Inn(G)\phi$, then $G$ has the $\phi$-twisted Rokhlin property if and only if it has the $\psi$-twisted Rokhlin property (Proposition \ref{lem:twistrokh}). In particular, this implies that the twisted Rokhlin property for an inner automorphism is equivalent to the Rokhlin property of the group. For Polish topological groups, we prove that the $\phi$-twisted Rokhlin property is equivalent to the property that, given  two non-empty open subsets $U$ and $V$ of $G$, there exists $g \in G$ such that $U \cap g V \phi(g)^{-1}\neq \emptyset$ (Proposition \ref{phi-TJEP}). Section \ref{trp of bmcg} contains our main results on the twisted Rokhlin property of big mapping class groups. By employing the approach outlined in the following flowchart, we establish Theorem \ref{theorem unique max end},  Proposition \ref{prop:2end}, Theorem \ref{thm:cantor} and Theorem \ref{thm:non-displaceable}.
$$\xymatrix{
	&*\txt{Infinite type surface} \ar[d] \ar[rd]  && \\
	& *\txt{Every compact subsurface\\is displaceable} \ar[dl] \ar[d] \ar[dr]& *\txt{There exists a compact \\non-displaceable subsurface} &\\
	*\txt{Unique maximal end}  & *\txt{Two maximal ends}  & *\txt{Space of maximal ends\\ is a Cantor space} &
}$$
\medskip

As a consequence, we deduce the following classification result.

\begin{theorem}
The mapping class group $\Mod(S)$ of a connected and orientable surface $S$ without boundary has the $\phi$-twisted Rokhlin property for some automorphism $\phi$ of $\Mod(S)$  if and only if the surface is either the $2$-sphere or satisfy the property that every compact subsurface of $S$ is displaceable and $S$ has a unique maximal end.
\end{theorem}

Along the way, we also investigate the $R_\infty$-property of mapping class groups of surfaces and prove the following result.

\begin{theorem}\label{theorem r-infinity property}
Let  $S$ be a connected orientable infinite-type surface without boundary. Then $\Mod(S)$ possesses the $R_\infty$-property.
\end{theorem}
\medskip

\section{Preliminaries}\label{section 2}
As the subject resides at the intersection of group theory, topology, and geometry, we offer relevant background information to enhance its readibility.

\subsection{Classification of surfaces}
A surface is said to be of {\it finite-type} if its fundamental group  is finitely generated; otherwise it is said  to be of {\it infinite-type}. Throughout the paper, our primary surface $S$ under consideration will be connected, orientable, without boundary and of infinite-type, unless stated otherwise. At some occasions, we shall also need finite-type surfaces, and we write $S_{g, n}^b$ to denote a finite-type surface of genus $g$ with $n$ punctures and $b$ boundary components.

\begin{definition}
An {\it exiting sequence} in a surface $S$ is a sequence $\{U_n\}_{n \ge 1}$ of connected open subsets of $S$ admitting the following properties:
\begin{enumerate}
\item If $m<n$, then $U_n \subset U_m$.
\item For each $n \in \mathbb N$, $U_n$ is not relatively compact.
\item For each $n \in \mathbb N$, $U_n$ has compact boundaries.
\item Each relatively compact subset of $S$ is disjoint from all except finitely many $U_n$'s.
\end{enumerate}
\end{definition}

Two exiting sequences $\{U_n\}_{n \ge 1}$  and $\{V_n\}_{n \ge 1}$  are {\it equivalent} if for every $n \in \mathbb N$,  there exists $m  \in \mathbb N$ such
that $U_m \subset V_n$ and $V_m \subset U_n$. An equivalence class of an exiting sequence is called an {\it end} and we denote the set of all ends by  $\mathcal{E}(S)$. The set $\mathcal{E}(S)$ can be  topologised as follows. Given a subset $U \subset S$ with compact boundary, consider the set 
$$U^*=\big\{[\{U_n\}_{n \ge 1}] \in \mathcal{E}(S) ~\mid~\textrm{there exists}~n \in \mathbb{Z}~\textrm{such that}~U_n \subset U \big\}$$
Then the desired topology on $\mathcal{E}(S)$ is the topology generated by the basis $$ \{ U^* ~\mid~ U \subset S~~ \text{is open with compact boundary}\}.$$ 
 With this topology, the space of ends of a surface is compact, totally disconnected, second countable and Hausdorff. Consequently, the space of ends is homeomorphic to a closed subset of the Cantor set.  If $U \subset S$ is an open subset with compact boundary and $ e \in  U^*$, then we refer to $U$ as a neighborhood of $e$.  Further, every homeomorphism $h: S\to S$ induces a homeomorphism  $\mathcal{E}(S)\to \mathcal{E}(S)$ given by $$[\{U_n \}_{n \ge 1}] \mapsto [\{h(U_n)\}_{n \ge 1}].$$
\par

An end is called {\it planar} if it has a neighbourhood that can be embedded in the plane. Otherwise, it is called {\it non-planar}. Let $\mathcal{E}_{np}(S)$ be the subspace of $\mathcal{E}(S)$ consisting of non-planar ends. Clearly,  $\mathcal{E}_{np}(S)$ is a closed set in $\mathcal{E}(S)$. The following theorem gives a classification of surfaces that we shall need \cite[Theorem 1]{MR0143186}.

\begin{theorem}
Let $S_1$ and $S_2$ be two surfaces. Let $g_1, g_2$ and $b_1, b_2$ represent the genus and the number of boundary components of $S_1, S_2$, respectively. Then $S_1$ is homeomorphic to $S_2$ if and only if $g_1 = g_2$, $b_1 = b_2$, and there exists a homeomorphism from $\mathcal{E}(S_1)$ to $\mathcal{E}(S_2)$ that restricts to a homeomorphism from $\mathcal{E}_{np}(S_1)$ to $\mathcal{E}_{np}(S_2)$.
\end{theorem}

Without loss of generality, we shall assume throughout that the surface $S$ is symmetric about the $xy$-plane when embedded in $\mathbb{R}^3$.  Since $\mathcal{E}(S)$ is homeomorphic to a closed subset of the Cantor set, it can be embedded on the great circle of intersection of the $xy$-plane and the 2-sphere $\mathbb{S}^2$. Now, take a sequence of disjoint closed disks $D_j$ on $\mathbb{S}^2 \setminus \mathcal{E}(S)$ which are symmetric about the $xy$-plane such that the set of accumulation points of this sequence is precisely $\mathcal{E}_{np}(S)$. We remove the interiors of the disks $D_j$ and suitably identify in pairs the boundary circles by orientation-reversing homeomorphisms such that the desired resulting surface is again symmetric about the $xy$-plane. See \cite[Theorem 3]{MR0143186} for the general construction. For instance, Figure \ref{fig:symsurface} illustrates  the construction of Loch Ness monster surface, where the infinite sequence of disjoint closed disks $\{D_i\}_{i\in \mathbb N}$ is chosen such that the radii of the disks decreases as it approaches a point of $\mathcal{E}(S)$.
\begin{figure}[H]
	\labellist
	\tiny
	\pinlabel $D_1$ at 80, 230
	\pinlabel $D_2$ at 190, 228
	\pinlabel $D_3$ at 100, 228
	\pinlabel $D_4$ at 168, 228
	\pinlabel $D_5$ at 125, 228
	\pinlabel $D_6$ at 148, 228
\pinlabel $x$ at 305, 118
	\pinlabel $y$ at 140, 330
	\endlabellist
	\centering
	\includegraphics[scale=0.57]{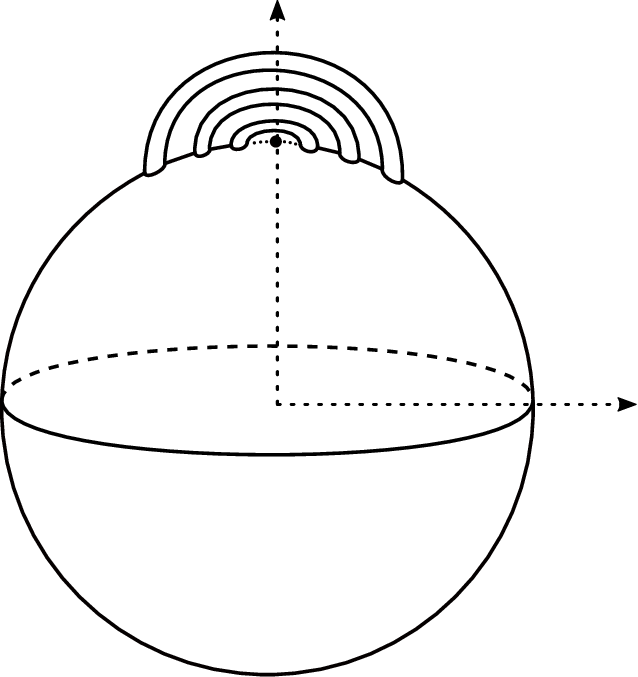}
	\caption{Realization of the Loch Ness monster surface.}
	\label{fig:symsurface}
\end{figure}
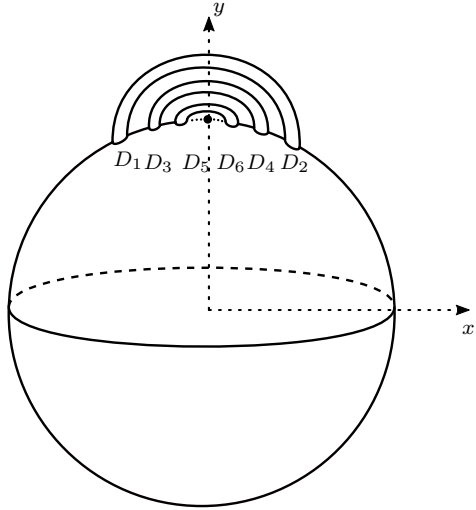

For a connected orientable surface $S$ (not necessarility of infinite-type), let $\operatorname{Homeo}^{+}(S)$ denote the group of orientation-preserving self-homeomorphisms of $S$. We define a binary relation $\preceq$ on the space of ends $\mathcal{E}(S)$ of $S$. For $x, y \in \mathcal{E}(S)$, we say that $y \preceq x$ if for every open neighbourhood $\mathcal{U}$ of $x$, there is an open neighbourhood $\mathcal{V}$ of $y$ and a homeomorphism $h \in \operatorname{Homeo}^{+}(S)$ such that $h(\mathcal{V}) \subseteq \mathcal{U}$.  An end $\mu$ is called {\it maximal} if, for any end $e$ satisfying $\mu \preceq e$, we must  have $e \preceq \mu$. Let $\mathcal{M}(S)$ denote the set of all maximal ends of $S$. Note that for every end $e \in \mathcal{E}(S)$, there exists a maximal end  $\mu \in \mathcal{M}(S)$ such that $e \preceq \mu$ \cite[Proposition 4.7]{MannRafi}. We say that two ends $e_1$ and $e_2$ are comparable if either $e_1 \preceq e_2$ or $e_2 \preceq e_1$. The space $\mathcal{E}(S)$ is said to be {\it self-similar} if  given a decomposition $\mathcal{E}(S) = \mathcal{E}_1\sqcup\cdots\sqcup\mathcal{E}_n$ into a finite disjoint union of clopen subsets, there exists an index $i \in \{1,\ldots, n\}$ such that $\mathcal{E}_i$ contains an open homeomorphic copy of $\mathcal{E}(S)$. The idea of  maximality and self-similarity for end spaces is due to \cite{MannRafi}.
\par

A subsurface $\Sigma$ of  $S$ is called {\it displaceable} if there exists a $h\in \operatorname{Homeo}^{+}(S)$ such that $\Sigma \cap h(\Sigma)=\emptyset$. We shall use the following result, which follows from \cite[Theorem 6.1]{MR4322618} and \cite{MannRafi}.

\begin{theorem}\label{compact subsurface is displaceable}
If $S$ is a surface in which every compact subsurface is displaceable, then either
\begin{enumerate}
\item $\mathcal{M}(S)$ has exactly one element,
\item $\mathcal{M}(S)$  has exactly two elements, or
\item $\mathcal{M}(S)$ is homeomorphic to a Cantor space in which any two maximal ends are comparable.
\end{enumerate}
Moreover, $\mathcal{M}(S)$ is  self-similar in the cases (1) and (3).
\end{theorem}
\medskip

\subsection{Mapping class groups of surfaces}

Recall that, the \textit{mapping class group}  $\Mod(S)$ of a connected orientable surface $S$ (not necessarily of infinite-type) is the group of isotopy classes of orientation-preserving self-homeomorphisms of $S$, which preserves the boundary of $S$ point-wise and permutes the punctures, if any. The mapping class group of an infinite-type surface is also referred as the {\it big mapping class group}. 
\par 

Let $\operatorname{Homeo}^{+}(S, \partial(S))$ be the group of orientation-preserving self-homeomorphisms of $S$ fixing the boundary $\partial(S)$ point-wise,
and equipped with the compact-open topology. We equip $\Mod(S)$ with the corresponding quotient topology inherited from $\operatorname{Homeo}^{+}(S,  \partial(S))$. Note that the mapping class group of a surface with compact boundary is discrete if and only if the surface is of finite-type. On the other hand, big mapping class groups have uncountably many elements and inherit a non-discrete topology from $\operatorname{Homeo}^{+}(S, \partial(S))$. The {\it extended mapping class group} of $S$ is the group $\Mod^{\pm}(S)$ of isotopy classes of, possibly orientation-reversing, self-homeomorphisms of $S$. We refer the reader to \cite{MR4264585} for a survey on both topological and algebraic aspects of big mapping class groups.
\par

For infinite-type surfaces without boundary, the outer automorphism group of the mapping class group is known due to \cite[Theorem 1.1]{MR4098634}. 

\begin{proposition}\label{outer auto mcg}
If $S$ is a connected orientable infinite-type surface without boundary, then 
$$\Out\left(\Mod^{\pm}(S)\right)=1 \quad \textrm{and} \quad \Out(\Mod(S)) \cong \mathbb{Z}_2.$$
\end{proposition}

As mentioned earlier, we assume that our surface $S$ is symmetric about the $xy$-plane when embedded in $\mathbb{R}^3$.  Thus, there is an orientation-reversing self-homeomorphism $$\m: S \rightarrow S$$ which maps a point on $S$ to its mirror image about the $xy$-plane.  Clearly, $\m$ is an involution of $S$. For notational convenience, let $\p$ denote the automorphism of $\Mod(S)$ given by $$\p(h )= \m h \m$$
 for all $h \in \Mod(S)$. Then $\p$ is a representative of the non-trivial element of $\Out(\Mod(S))$.
\par   
For a simple closed curve $c$ on a surface $S$, let $T_c$ denote the left-hand Dehn twist along $c$. By abuse of notation, we denote a curve and its isotopy class by the same symbol, and use the same convention for homeomorphisms as well. The following lemma, perhaps well-known, gives a relation between the Dehn twist and its conjugate by a homeomorphism of $S$. The proof of the following lemma is a slight modification of the arguments of \cite[Fact 3.7]{MR2850125}. See also \cite[Lemma 2.1]{MR2135748}.

\begin{lemma}
	\label{lem:twistprop1}
If $h$ is a homeomorphism of a surface $S$ and $c$ a simple closed curve on $S$, then
$$
h T_c h^{-1} = T_{h(c)}^{\epsilon},
$$
where $\epsilon=1$ when $h$ is orientation-preserving and $\epsilon=-1$ when $h$ is orientation-reversing.
\end{lemma}

By an {\it essential curve} on a surface $S$, we mean a simple closed curve on $S$ which does not bound a disk, a punctured disk, or an annulus. Let $C(S)$ be the set of isotopy classes of all essential simple closed curves on $S$. Given a subset $A$ of $C(S)$, we define the subset $U_A$ of $\Mod(S)$ by 
$$U_A = \{ \phi \in \Mod(S)~\mid~ \phi(a) = a \text{ for all } a \in A \}.$$
It is known that if $S$ has empty boundary and $\Mod(S)$ has trivial center, then the set 
$$\mathcal{B}=\{ \psi U_A ~\mid~ \psi \in \Mod(S)~~  \textrm{and}~~ A ~~  \text{is a finite subset of}~~C(S) \}$$
 is a basis for the topology on $\Mod(S)$. Note that, the center of the big mapping class group is always trivial \cite[Proposition 2]{MR4115214}. The following basic observation shall be used later.

\begin{lemma}\label{lemma:h=f}
Suppose that $g, h \in \Mod(S)$. Then $h \in g U_A$ if and only if $h U_A = g U_A$.
\end{lemma}

\begin{proof}
If $h \in g U_A$, then $h = g \phi$ for some $\phi \in U_A$. Thus, 
$g = h \phi^{-1} \in h U_A$, and hence $g U_A = h U_A$. Conversely, if $h U_A = g U_A$, then $h \in g U_A $.
\end{proof}
\medskip

\subsection{Curve graphs of surfaces}
We say that a simple closed curve $c$ on a surface $S$ is {\it non-separating} if $S \setminus c$ is connected; otherwise, it is called {\it separating}. The {\it intersection number} $ i(c,c')$ between two curves $c, c'$ on $S$ is defined to be the minimum number of points of intersections between the representatives in their isotopy classes. A {\it multicurve} on $S$ is a set of pairwise (non-isotopic) disjoint curves. 
\par

The set $C(S)$ can be viewed as a graph whose vertex set is $C(S)$ and there is an edge between two vertices if the corresponding curves on $S$ are disjoint. The graph $C(S)$ is referred as {\it curve graph} of $S$. There is a natural metric $d_{C(S)}$ on $C(S)$ by defining the length of each edge as 1. We shall need the following result \cite[Proposition 4.6]{MR1714338}.

\begin{proposition}\label{prop:anosovbound}
Let $S_{g,n}$ be a surface of genus $g$ with $n$ punctures such that $3g + n \ge 5$. Then there exists a real number $c > 0$ such that, for any pseudo-Anosov $h \in \Mod(S_{g,n})$,  $\gamma\in C(S_{g,n})$ and  $k\in \mathbb Z$, we have
$$
d_{C(S_{g,n})}(h^k(\gamma), \gamma) \geq c |k|.
$$
In particular, for any pseudo-Anosov $h \in \Mod(S_{g,n})$ and $\gamma\in C(S_{g,n})$, there exists $\ell \in \mathbb N$ such that
$$
d_{C(S_{g,n})}(h^m(\gamma), \gamma) > 2
$$
for all $m\geq \ell$.
\end{proposition}

We say that a subsurface $\Sigma$ and a simple closed curve $c$ on a surface $S$ have a \textit{non-trivial geometric intersection} if there exist $a \in C(\Sigma)$ such that $i(a,c) \neq 0$.

\begin{definition}
Let $\Sigma$ be a subsurface of a surface $S$ such that the embedding of $\Sigma$ into $S$ induces an embedding of $\mathcal{C}(\Sigma)$ into $\mathcal{C}(S)$. Let $c \in \mathcal{C}(S)$ be such that it has a non-trivial geometric intersection with $\Sigma$. Assume that the boundary $\partial (\Sigma)$ and $c$ are in a minimal position. Then a {\it projection} of $c$ onto $\mathcal{C}(\Sigma)$ is defined to be any component $b \in \mathcal{C}(\Sigma)$ of the boundary of a regular neighborhood of $\alpha \cup \partial(\Sigma)$, where $\alpha$ is a component of $c \cap \Sigma$.
\end{definition}

For example, let $S$ be a surface and $\Sigma$ be a genus two subsurface of $S$ with a single boundary component. Let $c$ be a curve on $S$ as illustrated in Figure~\ref{fig:projection}. Then the curves $b_1$ and $b_2$ are two projections of $c$ onto $\Sigma$  (see Figure~\ref{fig:projection}).
		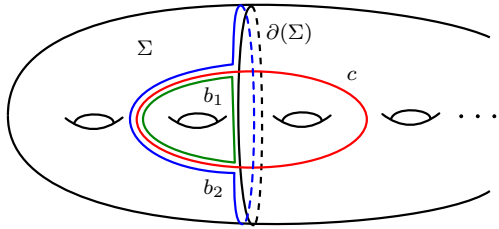
\begin{figure}[H]
	\labellist
	\tiny
	\pinlabel $\partial(\Sigma)$ at 205, 140
	\pinlabel $c$ at 250, 110
	\pinlabel $\Sigma$ at 100, 130
	\pinlabel $b_1$ at 150, 95
	\pinlabel $b_2$ at 150, 28
	\endlabellist
	\centering
	\includegraphics[scale=0.53]{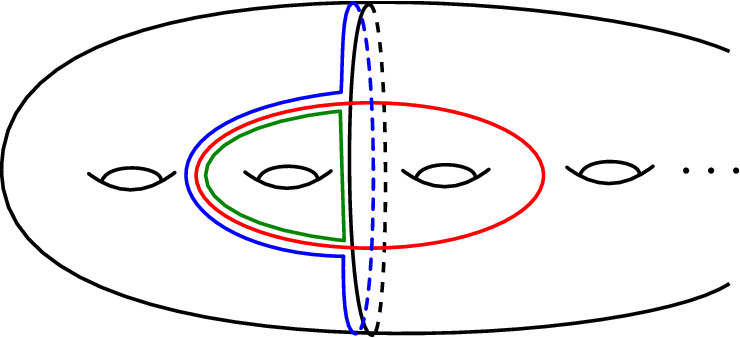}
	\caption{Projections of the curve $c$ on $\Sigma$.}
	\label{fig:projection}
\end{figure}

Note that, if each component of $\partial (\Sigma)$ is an essential separating simple closed curve on $S$, then  $\mathcal{C}(\Sigma)$ embeds into $\mathcal{C}(S)$. Further, for each $c \in \mathcal{C}(S)$ such that it has a non-trivial geometric intersection with $\Sigma$, such a projection $b \in \mathcal{C}(\Sigma)$ always exists, and if $c \in \mathcal{C}(\Sigma)$, then $b=c$.

\begin{lemma}\cite[Lemma 2.2]{MR1791145}\label{prop:projectionbound}
If $c \in \mathcal{C}(S)$ and $b,b^{\prime} \in  \mathcal{C}(\Sigma)$ are two projections of $c$, then $d_{\mathcal{C}(\Sigma)}\left(b, b^{\prime}\right) \leq 2$.
\end{lemma}
\medskip

\section{Twisted Rokhlin property}\label{section 3}

Recall that, a topological group is said to have {\it Rokhlin property} if it contains a dense conjugacy class.  Equivalently,  the topological group has a dense orbit under its conjugation action on itself. See, for example, the introduction of \cite{MR2308230} for a survey on Rokhlin property.

\begin{example}
Some topological groups with the Rokhlin property are:
\begin{enumerate} 
\item The symmetric group on $\mathbb{N}$.
\item The homeomorphism group of the Cantor set \cite[Theorem 2.6]{MR1867311}.
\item The group of orientation-preserving homeomorphisms of an even-dimensional sphere \cite[Theorem 3.4]{MR1867311}.
\item The (big) mapping class group of a connected orientable surface which is either the 2-sphere or non-compact with genus either zero or infinite and whose end space is self-similar with a unique maximal end \cite[Theorem 6.1]{MR4492497}.
\end{enumerate}
\end{example}

\begin{example}
Some topological groups without the Rokhlin property are: 
\begin{enumerate}
\item A non-trivial closed subgroup of $GL_n(\mathbb C)$ \cite[Lemma 5.2]{MR3415606}.
\item A non-trivial connected locally compact group  \cite[Theorem 5.3]{MR3415606}.
\end{enumerate}
\end{example}

Next we consider the idea of twisted conjugacy.

\begin{definition}
Let $G$ be a group and $\phi$ an automorphism of $G$. Two elements $x, y \in G$ are said to be $\phi$-twisted conjugate, denoted by $x \sim_\phi y$, if $y = gx\phi(g)^{-1}$ for some $g\in G$.   The resulting equivalence classes are called as {\it $\phi$-twisted conjugacy classes}, and we denote such a class by $C_\phi(x)$.
\end{definition}

Note that, given an automorphism $\phi$ of a group $G$, the map $G\times G\to G$ given by $(g,h) \mapsto gh\phi(g)^{-1}$ defines a left action of $G$ on itself.  Thus, $\phi$-twisted conjugacy classes are simply orbits under this action.
\begin{definition}
Let $G$ be a topological group and $\phi$ an automorphism of the underlying abstract group $G$. We say that $G$ has the {\it $\phi$-twisted Rokhlin property} if $G$ admits a dense $\phi$-twisted conjugacy class. Further, $G$ is said to have the {\it twisted Rokhlin property} if it has  $\phi$-twisted Rokhlin property for each automorphism $\phi$ of the underlying abstract group $G$.
\end{definition}

The subsequent proposition significantly simplifies the task of verifying the twisted Rokhlin property for a topological group.

\begin{proposition} \label{lem:twistrokh}
Let $G$ be a topological group, and $\phi, \psi$ be two automorphisms of $G$ such that $\psi \in \Inn(G)\phi$. Then $G$ has the $\phi$-twisted Rokhlin property if and only if it has the $\psi$-twisted Rokhlin property. In particular, if $\psi \in \Inn(G)$, then $G$ has the $\psi$-twisted Rokhlin property if and only if it has the Rokhlin property.
\end{proposition}

\begin{proof}
Let $\psi =\hat{g} \, \phi $, where $\hat{g}$ is the inner automorphism of $G$ induced by the element $g\in G$. We have
\begin{align*}
C_{\phi}(x) &= \{k x \phi(k)^{-1} \mid k \in G\}\\
&=\{ k x g^{-1} (g \phi(k)^{-1} g^{-1}) g \mid k\in G \}\\
&=\{ k (x g^{-1}) \,\hat{g} (\phi(k)^{-1}) g \mid k \in G \}\\
&=\{k (x g^{-1}) \,\psi(k)^{-1} g \mid k \in G\}\\
&=R_g(C_{\psi}(xg^{-1} )),
\end{align*}
where $R_g:G \to G$ is right multiplication by $g$. Since $R_g$ is a homeomorphism of $G$, it follows that $C_\phi(x)$ is dense in $G$ if and only if $C_{\psi}(xg^{-1})$ is dense in $G$.
\end{proof}

\begin{remark}
The preceding lemma implies that if $\Aut(G) = \Inn(G)$, then $G$ has  the Rokhlin property if and only if $G$ has the twisted Rokhlin property. In particular, if $G$ is the symmetric group on a countably infinite set, then every automorphism of $G$ is inner, and hence $G$ has the twisted Rokhlin property.
\end{remark}

\begin{proposition}	\label{lem:open}
Let $G$ be a topological group admitting a proper open subgroup $H$ such that $gH\phi(g)^{-1} = H$ for all $g \in G$. Then $G$ does not have the $\phi$-twisted Rokhlin property. In particular, if there exists a proper open normal subgroup $H$, then $G$ does not have the Rokhlin property. 
\end{proposition}

\begin{proof}
The given condition implies that,   for each $x \in G$, either $C_{\phi}(x) \subseteq H$ or $C_{\phi}(x) \cap H = \emptyset$. Thus, $C_{\phi}(x) \cap gH = \emptyset$ for some $g \in G$, where $gH$ is open. Therefore, $C_{\phi}(x)$ is not dense. Hence, $G$ does not have $\phi$-twisted Rokhlin property.
\end{proof}

\begin{remark}
By Proposition \ref{outer auto mcg}, if $S$ is an infinite-type surface, then  $\Aut(\Mod^{\pm}(S)) = \Inn(\Mod^{\pm}(S))$. Since $\Mod(S)$ is a proper open normal subgroup of $\Mod^{\pm}(S)$ (being a finite index subgroup), it follows that $\Mod^{\pm}(S)$ does not have the Rokhlin property. Consequently,  it does not have the twisted Rokhlin property. 
\end{remark}

\begin{definition}
Let $G$ be a topological group and let $\phi$ be an automorphism of $G$.  Then $G$ has the {\it  $\phi$-twisted joint embedding property}  ($\phi$-TJEP for short) if given any two non-empty open subsets $U$ and $V$ of $G$, there exists $g \in G$ such that $U \cap g V \phi(g)^{-1}\neq \emptyset$.
\end{definition}

Note that the set $g V \phi(g)^{-1}$ is open whenever $V$ is open. Recall that, a topological group is {\it Polish} if its underlying topological space is separable and completely metrizable. The next result extends \cite[Theorem 2.2]{MR4492497} to the $\phi$-twisted case.

\begin{proposition}\label{phi-TJEP}
Let $G$ be a Polish topological group and let $\phi$ be an automorphism of $G$. Then $G$ has the $\phi$-twisted Rokhlin property  if and only if $G$ has the $\phi$-TJEP.
\end{proposition}

\begin{proof}
Suppose that $G$ has the $\phi$-TJEP. Since $G$ is Polish,  it is second countable. Let $\mathcal{B}$ be a countable basis for $G$. For each $U \in \mathcal{B}$, define
$$
 D_U = \cup_{g\in G} ~g U \phi(g)^{-1}.
$$
Note that $D_U$ is open.  Further, it follows from the $\phi$-TJEP that $D_U$ is dense in $G$.  Since $G$ is completely metrizable, it is a Baire space. Hence, the set $D = \cap_{U\in \mathcal{B}} ~D_U$ is dense in $G$. Let $x\in D$ be an arbitrary element. Then, for each basis member $V\in \mathcal{B}$, there exists $g\in G$ (depending on $V$) such that $x \in g V \phi(g)^{-1}$. In other words, there exists $v \in V$ such that $v = g^{-1}x \phi(g)$, and hence $v \in V \cap C_\phi(x)$. This proves that $C_\phi(x)$ is dense in $G$.
\par 
Conversely, suppose that the $\phi$-twisted conjugacy class $C_\phi(x)$ is dense in $G$. Given non-empty open sets $U$ and $V$ in $G$, there exist $y, z \in G$ such that $y x \phi(y)^{-1}\in U$  and $z x \phi(z)^{-1}\in V$. Setting $h= y z^{-1}$, it is easy to see that $y x \phi(y)^{-1} \in U \cap h V \phi(h)^{-1}$. Hence, the group $G$ has the $\phi$-TJEP.
\end{proof}

\begin{remark}
It follows from \cite[Corollary 5]{Vlamisnotes}  that if $S$ is an infinite-type surface, then $\Mod(S)$ is a Polish group. Hence, $\Mod(S)$ has the $\phi$-twisted Rokhlin property if and only if it has the $\phi$-TJEP.
\end{remark}
\medskip

\section{Twisted Rokhlin property of big mapping class groups}\label{trp of bmcg}
It is known that the mapping class group of a surface with compact boundary (including no boundary) is discrete if and only if the surface is of finite-type. Thus, the mapping class group of a finite-type surface does not contain any proper dense subset, and consequently no dense $\phi$-twisted conjugacy class. Hence, we consider only infinite-type surfaces (without boundary).
\par

Recently, in \cite[Theorem 6.1]{MR4492497},  Lanier and Vlamis have shown that the big mapping class group of a connected orientable infinite-type surface without boundary has the Rokhlin property if and only if the surface is non-compact whose genus is either zero or infinite and whose end space is self-similar with a unique maximal end.
\par

If a topological group has the twisted Rokhlin property, then it has the Rokhlin property. Thus, we need to investigate the twisted Rokhlin property only for a non-compact surface whose genus is either zero or infinite and whose end space is self-similar with a unique maximal end. Further, in latter sections, we shall also see whether the big mapping class group of the remaining surfaces have the $\phi$-twisted Rokhlin property for some automorphism $\phi$. 

\medskip

\subsection{When the surface has a unique maximal end}
In this subsection, we assume that every compact subsurface of $S$ is displaceable and $\mathcal{M}(S)$ is a singleton set. Let $\mu$ be the unique maximal end.
Notice that, in this case, $S$ is either planar or has infinite genus, and the end space  $\mathcal{E}(S)$ is self-similar. Consider a separating simple closed curve $c$ on the surface $S$. Define $\Omega_c$ to be the component of $S \setminus c$ such that $\mu\in{\Omega}_c^*$ and $\Sigma_c$ to be the complement of $\Omega_c$ in $S$. Note that $\Sigma_c$ has $c$ as its unique  boundary component. Let $G$ be the subset of $\Mod(S)$ consisting of elements admitting representatives that restricts to the identity on $\Omega_c$ for some separating simple closed curve $c$. Note that the set
$$
\left\{{\Omega}_c^* ~\mid~ c~~ \textrm{is a separating simple closed curve on}~ S\right\}
$$
is a neighbourhood basis for $\mu$, and consequently $G$ is a subgroup of $\Mod(S)$. Since the surface $S$ is assumed to be symmetric about the $xy$-plane when embedded in $\mathbb{R}^3$, we have the orientation-reversing self-homeomorphism $\m: S \rightarrow S$. Given a subsurface $\Sigma$ of $S$, the intersections 
$$\Sigma^+:=\Sigma \cap \{(x,y,z)\in \mathbb{R}^3~\mid~ z \geq 0\} \quad \textrm{and} \quad \Sigma^-:=\Sigma \cap \{(x,y,z)\in \mathbb{R}^3~\mid~ z \leq 0\}$$ 
are called the positive and the negative parts of $\Sigma$, respectively. We now present an extension of \cite[Proposition 4.2]{MR4492497}. 

\begin{proposition}\label{prop:uniquemax}
Let $S$ be a surface in which every compact subsurface is displaceable and $S$ has a unique maximal end, say $\mu$. Let $G$ be the subgroup of $\Mod(S)$ defined above. Then the following assertions holds:
\begin{enumerate}
\item The subgroup $G$ is dense in $\Mod(S)$.
\item For each separating simple closed curve $c$ in $S$ such that $\m(\Sigma_c) = \Sigma_c$, there exists $g \in \Homeo^+(S)$ such that $g(\Sigma_c) \subset \Omega_c$ and $g \m g^{-1}\m$ is identity on $\Sigma_c$.
	\end{enumerate}
\end{proposition}

\begin{proof}
Assertion (1) follows from \cite[Proposition 4.2]{MR4492497}. Let $c$ be a separating simple closed curve in $S$ such that $\m (\Sigma_c) = \Sigma_c$. Since $S$ is of infinite-type, we can choose another separating simple closed curve $c^\prime$ such that $\Sigma_{c}\cap \Sigma_{c^\prime} =\emptyset$,  $\m (\Sigma_{c^\prime}) = \Sigma_{c^\prime}$ and $\Sigma_c\cong \Sigma_{c^\prime}$. Further, choose a separating simple closed curve $b$ such that the curves $\{c, c^{\prime}, b\}$ bounds a pair of pants (a subsurface homeomorphic to the two-sphere with three disks removed) and $\{c,c'\} \subset \Sigma_b$. Fix some hyperbolic metric on $\Sigma_b$ and choose an isometry $h:\Sigma_b \to \Sigma_b$ fixing $\partial(\Sigma_b)$ identically such that $h(\Sigma_c)=\Sigma_{c^\prime}$ and $h(\Sigma_c^+)=\Sigma_{c^\prime}^+$.  Define $f:S \to S$ by
	$$f(x) = \begin{cases}
	h(x) & \text{ if } x \in \Sigma_c^+,\\
	\m  h \m(x) & \text{ if } x \in \Sigma_c^-,\\
	h(x) & \text{ if } x \in \Sigma_b \setminus \Sigma_{c},\\
	x & \text{ if } x \in S \setminus \Sigma_b. 
	\end{cases}$$
	Thus $g:=f^{-1} \in \Homeo^+(S)$ such that $g^{-1}(\Sigma_{c}) \subset \Omega_c$. Consequently, we have $g(\Sigma_{c}) \subset \Omega_c$. Now, if $x\in \Sigma_c^{+}$, then $\m(x) \in \Sigma_c^{-}$, and hence
$$g\m g^{-1} \m (x) = g \m \m  h \m (\m(x)) =g h(x)=g f(x)=x.$$
Similarly, if $x\in \Sigma_c^{-}$, then $\m(x) \in \Sigma_c^{+}$, and hence $ g\m g^{-1} \m (x) = x.$ This shows that $g \m g^{-1}\m$ is identity on $\Sigma_c$, which establishes assertion (2).
	\end{proof}

\begin{theorem}\label{theorem unique max end}
Let $S$ be a surface in which every compact subsurface is displaceable and $S$ has a unique maximal end. Then $\Mod(S)$ has the twisted Rokhlin property.
\end{theorem}

\begin{proof}
It follows from \cite[Theorem 2.2 and Lemma 4.3]{MR4492497} that $\Mod(S)$ has the Rokhlin property. Consequently, by Proposition \ref{lem:twistrokh},  $\Mod(S)$ has the $\phi$-twisted Rokhlin property for each inner automorphism $\phi$ of $\Mod(S)$. By Proposition \ref{outer auto mcg}, we have $\Out(\Mod(S)) \cong \mathbb{Z}_2$. Thus, in view of 
Proposition \ref{lem:twistrokh} and Proposition \ref{phi-TJEP}, it is enough to prove that $\Mod(S)$ has the $\p$-TJEP.
\par

Let $U_1$ and $U_2$ be non-empty open subsets of $\Mod(S)$. Without loss of generality, we assume that $U_1$ and $U_2$ are the basis elements of $\Mod(S)$. Thus, there exist finite subsets $A_1, A_2 \subset C(S)$ and $h_1, h_2 \in \Mod(S)$ such that $U_1 = h_1U_{A_1}$ and $U_2 = h_2U_{A_2}$. By Proposition \ref{prop:uniquemax}(1), the subgroup $G$ is dense in $\Mod(S)$. Thus, by Lemma \ref{lemma:h=f}, we can assume that $h_1,h_2 \in G$.  Since $h_1, h_2 \in G$, there exist separating simple closed curves $c_1,c_2$ on $S$ such that $h_1$ and $h_2$ admit representatives that restrict to the identity on $\Omega_{c_1}$ and $\Omega_{c_2}$, respectively. Next, we choose a separating simple closed curve $b$ such that $c_1, c_2 \subset \Sigma_b$ and $\m(\Sigma_b)=\Sigma_b$. Moreover, by enlarging $\Sigma_b$, if necessary, we can assume that each curve in $A_1 \cup A_2$ has a representative in $\Sigma_b$ and $\m (\Sigma_b) = \Sigma_b$. By Proposition \ref{prop:uniquemax}(2), there exists $g  \in \Homeo^+(S)$ such that $g(\Sigma_b) \subset \Omega_b$ and $g\m g^{-1}\m$ is identity on $\Sigma_b$. Note that $g(\Sigma_b) \subset \Omega_b$ implies that $g^{-1}(\Sigma_b) \subset \Omega_b$. We have 
 \begin{equation}\label{gu2phig}
gU_2\p(g)^{-1} = gh_2U_{A_2}\p(g)^{-1} = gh_2\p(g)^{-1}U_{\p(g)(A_2)}.
\end{equation}
 The element $h_1$ has a representative that restricts to the identity on $\Omega_b$. Similarly, since $g\m =\m g$ on $\Sigma_b$ and $g^{-1}(\Sigma_b) \subset \Omega_b$, it follows that the element $gh_2\p (g)^{-1}$ also has a representative that restricts to the identity on $\Sigma_b$. Thus, we have the following:
\begin{itemize}
\item $h_1$ and $gh_2\p (g)^{-1}$ commute on $S$,
\item $h_1 \in U_{\p(g)(A_2)}$ (since $\p (g)(A_2) \subset \p (g)(\Sigma_b) \subset  \Omega_b$),
\item $gh_2\p (g)^{-1} \in U_{A_1}$ (since $A_1 \subset \Sigma_b$).
\end{itemize}
 Taking $h = h_1(gh_2\p (g)^{-1})$, we see that $$h \in h_1U_{A_1} = U_1$$ and $$h \in gh_2\p(g)^{-1}U_{\p(g)(A_2)} = gU_2\p(g)^{-1}\quad (\textrm{by}~\eqref{gu2phig}).$$ 
  Thus, $h \in U_1 \cap gU_2 \p(g)^{-1}$, which completes the proof.
\end{proof}

\subsection{When the surface has exactly two maximal ends}
In this subsection, we assume that every compact subsurface of our surface $S$ is displaceable and $\mathcal{M}(S)$ has exactly two elements. We claim that $\Mod(S)$ does not have the $\phi$-twisted Rokhlin property for any $\phi \in \Aut(\Mod(S))$. It is enough to consider the automorphism $\p$ of $\Mod(S)$. Recall that, there is an action of $\Mod(S)$ on $\mathcal{E}(S)$ that keeps  $\mathcal{M}(S)$ invariant.

\begin{lemma}\label{lem:stab}
The stabiliser $\Stab(\mu)$ of a maximal end is a proper open normal subgroup of $\Mod(S)$ such that $h\Stab(\mu)\p(h)^{-1} = \Stab(\mu)$ for all $h \in \Mod(S)$.
\end{lemma}

\begin{proof}
It follows from \cite[Lemma 5.4]{MR4492497} that $\Stab(\mu)$ is a proper open normal subgroup of $\Mod(S)$. If $f \in h\Stab(\mu)\p (h)^{-1}$, then $f= h f' \m h^{-1} \m$ for some $f' \in \Stab(\mu)$. Since $\Stab(\mu)$ is normal,  we see that 
$$f(\mu)= h f' \m h^{-1} \m(\mu)=h f' \m h^{-1} (\mu)=h f' h^{-1} (\mu)=\mu.$$
Hence,  $h\Stab(\mu)\p(h)^{-1} \subseteq \Stab(\mu)$ for all $h \in \Mod(S)$. The converse inclusion can be proved in a similar way.
\end{proof}

\begin{proposition}\label{prop:2end}
Let $S$ be a  surface in which every compact subsurface is displaceable and $\mathcal{M}(S)$ has exactly two elements. Then $\Mod(S)$ does not have the $\phi$-twisted Rokhlin property for any $\phi \in \Aut(\Mod(S))$.
\end{proposition}

\begin{proof}
The proof  follows directly from Proposition \ref{lem:twistrokh}, Proposition \ref{lem:open} and Lemma \ref{lem:stab}.
\end{proof}
\medskip

\subsection{When the space of maximal ends is a Cantor space}
This subsection deals with the case when every compact subsurface of $S$ is displaceable and $\mathcal{M}(S)$ is a Cantor space. The following result is a special case of \cite[Theorem 14.1]{MR2850125}.

\begin{proposition}\label{prop:anosov}
Let $S_{g,n}$ be a surface of genus $g$ with $n$ punctures and  let $\{a,b\}$ be a filling system of curves for $S_{g,n}$. Consider the subgroup  $\langle T_a, T_b \rangle$ of $\Mod(S)$
and the representation $\rho: \langle T_a, T_b \rangle \to PSL(2, \mathbb R)$ given by 
		$$
		\rho(T_a)= \left(\begin{matrix}
		1 & i(a,b) \\
		0 & 1
		\end{matrix}\right) \quad \textrm{and} \quad \rho(T_b )= \left(\begin{matrix}
		1 & 0 \\
		-i(a,b) & 1
		\end{matrix}\right).
		$$
Let  $f\in \langle T_a, T_b \rangle$ such that $\rho(f)$ is hyperbolic in $PSL(2,\mathbb R)$. Then $f$ is pseudo-Anosov.
\end{proposition}

The following lemma will be used in the main result of this subsection.

\begin{lemma}\label{pApermute}
Consider the surface $S_{0,6}$ such that it is symmetric about the $xy$-plane and all its punctures lie on the $xy$-plane. Then there exists a mapping class $f \in \Mod(S_{0,6})$ such that $f$ permutes the punctures cyclically and $f\m f\m$ is pseudo-Anosov.
\end{lemma} 

\begin{proof}
Let $b_1, \ldots, b_6$ be the punctures on $S_{0,6}$. Let $r$ be the order six homeomorphism of $S_{0,6}$ such that $r(b_i)=b_{i+1}$ for $1 \leq i \leq 5$ and $r(b_6) = b_1$. Let $\m$ denote a map that sends a point of $S_{0,6}$ to its  mirror image about the $xy$-plane, $c$ be the simple closed curve on $S_{0,6}$ as shown in Figure \ref{fig:curve} and $d= \m(c)$. Then $f = r T_c T_{d}^{-2} T_{c}^3 T_{d}^{-4}$ is an orientation-preserving homeomorphism of $S_{0,6}$. Note that $r(c) = \m (c)=d$ and $r \m = \m r$ on $S_{0,6}$. Furthermore,  $f(b_i) = b_{i+1}$ for $1 \leq i \leq 5$ and $f(b_6) = b_1$.
	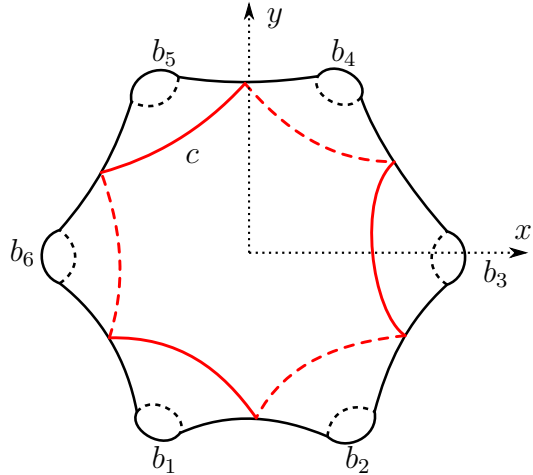
\begin{figure}[H]
		\labellist
		\pinlabel $c$ at 110, 210
		\pinlabel $x$ at 350, 155
		\pinlabel $y$ at 170, 310
		\pinlabel $b_1$ at 90, -10
		\pinlabel $b_2$ at 230, -10
		\pinlabel $b_3$ at 330, 125
		\pinlabel $b_4$ at 220, 285
		\pinlabel $b_5$ at 90, 285
		\pinlabel $b_6$ at -13, 140
		\endlabellist
		\centering
		\includegraphics[scale=0.53]{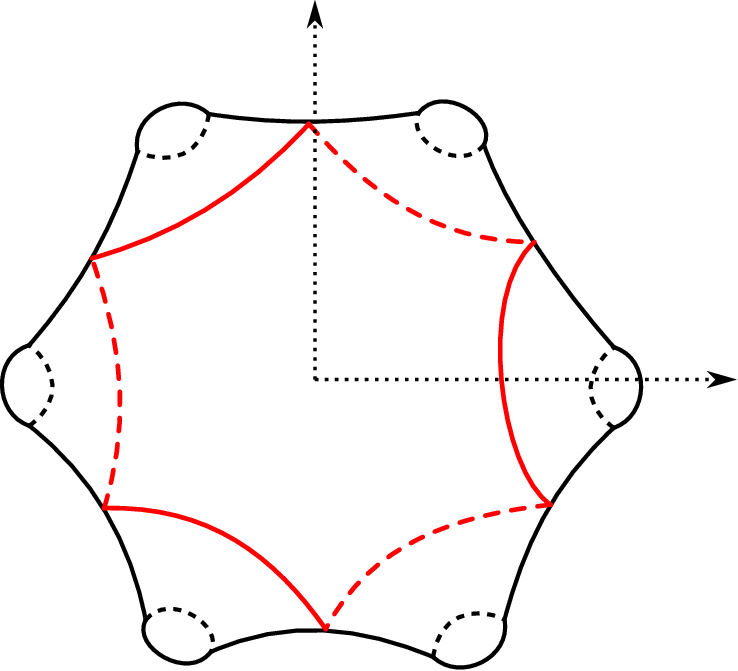}
		\caption{The curve $c$ on $S_{0,6}$ such that $\{ c, \m(c)\}$ fills $S_{0,6}$.}
		\label{fig:curve}
	\end{figure}
We see that
	\begin{align*}
	(f\m f\m)^3 &= (r T_c T_{d}^{-2} T_{c}^3 T_{d}^{-4} r T_{d}^{-1} T_{c}^{2} T_{d}^{-3} T_{c}^{4})^3, ~~\text{since}~~  r \m = \m r \text{ and } \m T_c \m = T_{\m (c)}^{-1}\\
	& = (T_{d} T_{c}^{-2} T_{d}^{3} T_{c}^{-4}T_{d}^{-1} T_{c}^{2} T_{d}^{-3} T_{c}^{4} r^2)^3,~~\text{since}~~ rT_c {r}^{-1} = T_{r(c)}= T_d\\
	& = (T_{d} T_{c}^{-2} T_{d}^{3} T_{c}^{-4}T_{d}^{-1} T_{c}^{2} T_{d}^{-3} T_{c}^{4})^3.
	\end{align*}
Since $i(c,d)=6$, by Proposition~\ref{prop:anosov}, we have $$
	\rho(T_c)=\left(\begin{matrix}
	1 & 6 \\
	0 & 1
	\end{matrix}\right) \quad \textrm{and} \quad \rho(T_d)=\left(\begin{matrix}
	1 & 0 \\
	-6 & 1
	\end{matrix}\right).
	$$
This gives
$$
\rho((f\m f\m)^3)= \left(\begin{matrix}
	-6768504263453271167830031 & -162818421281863319251229760 \\
	41172474347137850845138560 & 990416347916847555740563729
	\end{matrix}\right),$$
	which is hyperbolic since the trace of $\rho \big((f\m f\m)^3 \big)$ is greater than 2. Hence, $f\m f\m$ is a pseudo-Anosov map.
\end{proof}

\begin{theorem}\label{thm:cantor}
Let $S$ be a surface in which every compact subsurface is displaceable and $\mathcal{M}(S)$ is homeomorphic to a Cantor space. Then $\Mod(S)$ does not have the $\phi$-twisted Rokhlin property for any $\phi\in \Aut(\Mod(S))$.
\end{theorem}

\begin{proof}
By Proposition \ref{lem:twistrokh} and \cite[Lemma 4.8]{MR4492497}, it is enough to prove that $\Mod(S)$ does not have the $\p$-twisted Rokhlin property. By results in \cite[Proposition 4.6]{MR4492497} (or \cite[Lemma 4.18]{MannRafi}), we choose a compact subsurface $\Sigma$ of $S$ such that
	\begin{enumerate}
		\item $\Sigma$ is planar,
		\item $\Sigma$ has six boundary components, say $b_1, b_2, b_3, b_4, b_5,b_6,$
		\item each $b_i$ is a separating curve of $S$,
		\item $\Omega_1, \Omega_2, \Omega_3, \Omega_4, \Omega_5, \Omega_6$ are the six components of $S \setminus \Sigma$ such that $\Omega_i^*$ is homeomorphic to $\mathcal{E}(S)$,
\item $\m (\Sigma) = \Sigma, ~ \m (b_i) = b_i$ and $\m (\Omega_i) = \Omega_i$ for each $1 \leq i \leq 6$.
\end{enumerate}
By Lemma~\ref{pApermute}, let $f^\prime \in \Homeo^{+}(\Sigma)$ such that $f^\prime(b_i) = b_{i+1}$ for $1 \leq i \leq 5$ and $f^\prime(b_6) = b_1$ and $f'\m f'\m$ is pseudo-Anosov. Using condition (4), there is an extension $f \in \Mod(S)$ of $f'$.
\par 
Recall that a subset $A$ of $C(S)$ is called a {\it stable Alexander system} if $U_A$ is the center of $\Mod(S)$.  Choose a stable Alexander system $A$ for $\Sigma$. Let us take $U = fU_A$ and $V = U_{\{b_1\}}$. We claim that $U \cap g V \p(g)^{-1} = \emptyset$ for all $g \in \Mod(S)$. For an element $g \in \Mod(S)$, we consider the following cases:
\par
Case 1: Suppose that $g(b_1)$ has a representative which is disjoint from $\Sigma$. In this case, $g(b_1) \in \Omega_i$ for some $i$. If $h \in g V\p(g)^{-1}$, then using condition (5), we obtain $\p(g)(b_1) \in \Omega_i$  and $h(\p(g)(b_1)) = g(b_1) \in \Omega_i$. Hence, we have $h(\Omega_i) \cap \Omega_i \neq \emptyset$. And, if $h \in U$, then $h(\Omega_i) \cap \Omega_i = \emptyset$ since $h(b_i) = b_{i+1}$. This proves that $U  \cap g V\p(g)^{-1} = \emptyset$.
\par

Case 2: Suppose that $g(b_1)$ has a non-trivial geometric intersection with $\Sigma$. Let $c \in C(\Sigma)$ be a projection of $g(b_1)$ onto $\Sigma$. If $h \in U \cap g V \p(g)^{-1}$, then $h\m g(b_1)= h(\p(g)(b_1)) = g(b_1)$. This gives  $$h\m h\m g(b_1)= h \m g(b_1) = g(b_1),$$  and consequently $(h\m h\m)^n (g(b_1)) = g(b_1)$ for all $n \in \mathbb{N}$. Since $h \in U$, it follows that $(h\m h\m)^n (c) \in C(\Sigma)$ is also a projection of $g(b_1)$ onto $\Sigma$ for all $n\in \mathbb N$. Hence, by  Lemma \ref{prop:projectionbound}, we have 
\begin{equation}\label{eqn:cantor1}
d_{C(\Sigma)}((h\m h\m)^n (c),c) \leq 2
\end{equation}
for all $n \in \mathbb N$. Since $h \in U$, we have $h\m h\m = f \m f \m$ on $\Sigma$, and hence the map $h\m h\m$ is pseudo-Anosov on $\Sigma$.  By Proposition \ref{prop:anosovbound}, there exists $m\in \mathbb N$ such that
 \begin{equation*}
d_{C(\Sigma)}\big((f\m f\m)^m(c), c \big) > 2 
 \end{equation*}
 for all $c \in C(\Sigma)$. This contradicts \eqref{eqn:cantor1}. Hence, we must have $U \cap g V \p(g)^{-1} = \emptyset$ for all $g \in \Mod(S)$. Hence, by Proposition \ref{phi-TJEP}, the proof of the theorem is complete.
\end{proof}
\medskip

\subsection{When the surface admits a compact non-displaceable subsurface}
Finally, we consider the case when $S$ admits a compact non-displaceable subsurface.

\begin{lemma}\label{lemma:fills1}
Consider the surface $S_{g,n}$ of genus $g \ge 0$ with $n$ punctures, where $n> 5$ is even when $g=0$. Suppose that $S_{g,n}$ is symmetric about the $xy$-plane with all the punctures lying on the $xy$-plane. Then there exists an essential simple closed curve $c$ on $S_{g,n}$ such that the collection $\{c, \m(c)\}$ fills the surface.
Moreover, there exists an orientation-preserving homeomorphism $f$ of $S_{g,n}$ such that $f\m f\m$ is pseudo-Anosov.
\end{lemma}

\begin{proof}
If $g>0$, then the curve $c$ as shown in Figure \ref{fig:fill} together with $\m(c)$ fills the surface $S_{g,n}$.

\begin{figure}[H]
	\centering
	\includegraphics[scale=0.53]{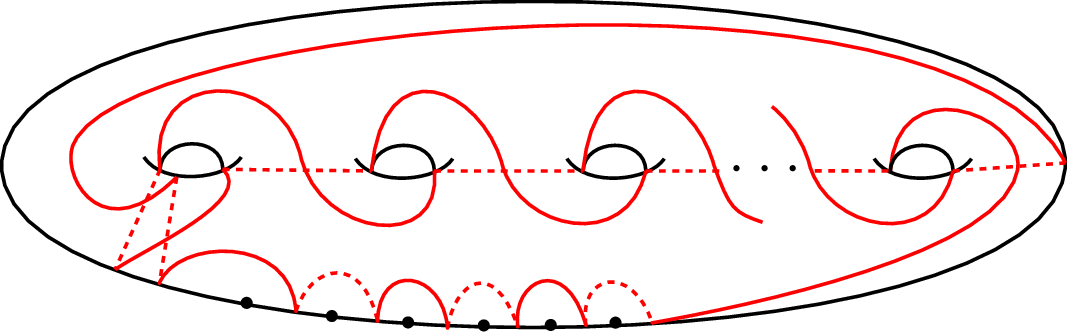}
	\caption{The curve $c$ on $S_{g,n}$ such that $\{ c,\m(c) \}$ fills $S_{g,n}$.}
	\label{fig:fill}
\end{figure}
If $g=0$, then choose the simple closed curve $c$ to be as depicted in Figure \ref{fig:curve} for $S_{0,6}$. Take $f = T_c$, where $c$ is the curve as above. By using Penner's construction \cite[Theorem 14.4]{MR2850125}, the homeomorphism $f\m f\m = T_c T_{\m(c)}^{-1}$ is pseudo-Anosov since the collection $\{c, \m(c)\}$ fills $S_{g,n}$. 
\end{proof}

\begin{theorem}\label{thm:non-displaceable}
Let $S$ be a surface containing a compact non-displaceable subsurface. Then $\Mod(S)$ does not have the $\phi$-twisted Rokhlin property for any $\phi\in \Aut(\Mod(S))$.
\end{theorem}

\begin{proof}
In view of Proposition \ref{lem:twistrokh} and \cite[Theorem 3.1]{MR4492497}, it is enough to prove that  $\Mod(S)$ does not have the $\p$-twisted Rokhlin property. 
Let $\Sigma$ be a compact non-displaceable subsurface of $S$. By enlarging $ \Sigma $, we can ensure that it is connected, has a stable Alexander system, satisfies $\m(\Sigma)=\Sigma$ and its boundary $\partial(\Sigma)$ consists of at least six components with an even number of boundary components, each being essential and separating. Furthermore, $\Sigma$ remains a finite-type subsurface of $S$. 
\par

We claim that there exist open subsets $U$ and $V$ of $\Mod(S)$ such that $U \cap g V\p(g)^{-1}=\emptyset$ for all $g \in \Mod(S)$. Let $A \subset C(S)$ be a stable Alexander system for $\Sigma$. It follows from the proof of Lemma \ref{lemma:fills1} that there is an orientation-preserving homeomorphism $f': \Sigma \to \Sigma$ which fixes $\partial(\Sigma)$ point-wise such that $f'\m f'\m$ is pseudo-Anosov. Thus, by Proposition \ref{prop:anosovbound}, there exists $m\in \mathbb N$ such that
 \begin{equation}\label{eq2}
d_{C(\Sigma)}\big((f'\m f'\m)^m(a), a \big) > 2 
 \end{equation}
 for all $a \in C(\Sigma)$. Let $f \in \Mod(S)$ denote the extension of $f'$. Take $U=fU_A$ and $V=U_A$. Suppose that there exists $g \in \Mod(S)$ such that $h \in f U_A \cap gU_A\p(g)^{-1}$. Since $h \in gU_A\p (g)^{-1}=g\p(g)^{-1}U_{\p(g)(A)}$, we have
$$
h \m g \m(c)= g(c)
$$
for all $c \in A$. Thus, we obtain
\begin{equation}\label{eq1}
h \m h \m (g(c)) = h \m h \m g \m \m(c) = h \m g (\m(c)) = g(c)
\end{equation}
for all  $c \in A$. Since $\Sigma$ is non-displaceable in $S$, there exists a curve $c \in C(\Sigma)$ such that $g(c)$ and $\Sigma$ have a non-trivial geometric intersection. Let $b$ denote a projection of $g(c)$ onto $\Sigma$. By \eqref{eq1}, 
$(h\m h \m)^n (b)$ is also a projection of $g(c)$ onto $(h \m h \m)^n (\Sigma)=\Sigma$. By Lemma \ref{prop:projectionbound}, we have
$$
d_{C(\Sigma)}((h\m h \m)^n (b), b)\leq 2
$$
for all $n\in \mathbb N$. Since $h \in f U_A$,  by \eqref{eq2}, we have
$$
d_{C(\Sigma)}((h\m h \m)^m (b), b) = d_{C(\Sigma)}((f\m f \m)^m (b), b) >2
$$
for some $m$, which is the contradiction. Hence, we conclude that $fU_A \cap gU_A\p(g)^{-1}=\emptyset$ for all $g \in \Mod(S)$. Hence, by Proposition \ref{phi-TJEP}, the proof of the theorem is complete.
\end{proof}
\medskip

The outcomes from this section lead to the following theorem.

\begin{theorem}
The mapping class group $\Mod(S)$ of a connected and orientable surface $S$ without boundary has the $\phi$-twisted Rokhlin property for some automorphism $\phi$ of $\Mod(S)$  if and only if the surface is either the $2$-sphere or satisfy the property that every compact subsurface of $S$ is displaceable and $S$ has a unique maximal end.
\end{theorem}

\begin{proof}
The mapping class group of the $2$-sphere obviously satisfies the desired property being a trivial group. If the surface $S$ is such that every compact subsurface of $S$ is displaceable and $S$ has a unique maximal end, then by Theorem~\ref{theorem unique max end}, it has the twisted Rokhlin property, and hence $\phi$-twisted Rokhlin property for each $\phi$. Conversely, suppose that $\Mod(S)$ has the  $\phi$-twisted Rokhlin property for some automorphism $\phi$ of $\Mod(S)$. If $S$ is compact, then it must be the 2-sphere. If $S$ is non-compact, then by Theorem~\ref{thm:non-displaceable}, every compact subsurface of $S$ must be displaceable. Finally, by Theorem~\ref{compact subsurface is displaceable}, Theorem~\ref{theorem unique max end}, Proposition~\ref{prop:2end}, and Theorem~\ref{thm:cantor}, $S$ must satisfy the property that every compact subsurface of $S$ is displaceable and $S$ has a unique maximal end.
\end{proof}
\medskip

\section{$R_{\infty}$-property of big mapping class groups} 
In this section, we investigate the number of twisted conjugacy classes in mapping class groups.

\begin{definition}
Let $G$ be a group and $\phi \in \Aut(G)$. The {\it Reidemeister number} of $\phi$, denoted by $R(\phi)$, is the number of $\phi$-twisted conjugacy classes in $G$. The group $G$ is said to have the {\it $R_\infty$-property} if $R(\phi)$ is infinite for each automorphism $\phi$ of $G$.
\end{definition}

The $R_{\infty}$-property for the mapping class groups of finite-type surfaces has been investigated in \cite[Theorem 4.3]{MR2644279}. It has been proven that the mapping class group of a closed orientable surface has the $R_{\infty}$-property if and only if the surface is not the 2-sphere. A partial result for infinite-type surfaces has been established in \cite[Theorem 3.4]{MR4525669}, for instance,  it is proven that if a connected orientable infinite-type surface without boundary admits a non-displaceable compact subsurface, then its mapping class group possesses the $R_{\infty}$-property. The aim of this section is to provide a complete answer for infinite-type surfaces through a general approach.
\par

Let $S$ be a connected orientable infinite-type surface without boundary.  Recall that, $\Out(\Mod(S)) \cong \mathbb Z_2$. In fact,  the group is generated by the coset of the restriction of the inner automorphism $\hat{f}$ of $\Mod^{\pm}(S)$ induced by an orientation reversing homeomorphism $f$ of $S$ of order two (see \cite[Theorem 1.1]{MR4098634} and \cite{MR0970079}). We use the same notation $\hat{f}$ to denote its restriction on $\Mod(S)$.
\par 

The subsequent lemma provides a connection between $\hat{f}$-twisted conjugacy classes in $\Mod(S)$ and conjugacy classes of $\Mod^{\pm}(S)$ lying in the coset $\Mod(S)f$. See \cite[Lemma 2.1]{MR2639839} for a general result.

\begin{lemma}\label{lemma:twistedconj1}
Let $S$ be a connected orientable infinite-type surface without boundary and $f$ an orientation reversing homeomorphism of $S$. Then, two elements $g, h \in \Mod(S)$ are $\hat{f}$-twisted conjugate in $\Mod(S)$ if and only if $gf$ and $hf$ are conjugate in $\Mod^{\pm}(S)$.
\end{lemma}

\begin{proof}
Suppose that there exists $r \in \Mod(S)$ such that $h = r g \hat{f}(r)^{-1} = r g f r^{-1} f^{-1}$. But, this gives $hf = r (gf) r^{-1}$, that is, $gf$ and $hf$ are conjugate in $\Mod^{\pm}(S)$.
\par
Conversely, suppose that there exists $s \in \Mod^{\pm}(S)$	such that $hf = s (gf) s^{-1}$. If $s\in \Mod(S)$, then
$h = s g f s^{-1} f^{-1} = s g \hat{f}(s)^{-1}$, and we are done. If $s\not\in \Mod(S)$, then $s=t\, f$ for some $t \in\Mod(S)$. In this case, we have
$$h = t (f g f) f^{-1} t^{-1} f^{-1} = t (fgf^{-1}) \hat{f}(t)^{-1}.$$
Hence, $h$ and $(f g f^{-1})$ are $\hat{f}$-twisted conjugate in $\Mod(S)$. Since $(f g f^{-1})= g^{-1} g(f g f^{-1})= g^{-1} g \hat{f}(g^{-1})^{-1}$, the elements 
$(f g f^{-1})$ and $g$ are $\hat{f}$-twisted conjugate in $\Mod(S)$. Consequently, $h$ and $g$ are $\hat{f}$-twisted conjugate in $\Mod(S)$, and the proof is complete.
\end{proof}

\begin{lemma} \label{prop:multicurve}
Let $S$ be a connected orientable surface.
\begin{enumerate}
\item  Let $\left\{a_1, \ldots, a_l\right\}$ and  $\left\{b_1, \ldots, b_k\right\}$ be multicurves on $S$. Let $p_1, \ldots, p_l$ and $q_1, \ldots, q_k$ be non-zero integers. If
	$$
	T_{a_1}^{p_1} T_{a_2}^{p_2} \cdots T_{a_l}^{p_l}=T_{b_1}^{q_1} T_{b_2}^{q_2} \cdots T_{b_k}^{q_k}
	$$
	in $\Mod(S)$, then $l=k$ and $\left\{T_{a_1}^{p_1}, \ldots, T_{a_l}^{p_l}\right\}=\left\{T_{b_1}^{q_1}, \ldots, T_{b_k}^{q_k}\right\}$ \cite[Lemma 3.17]{MR2850125}.
\item Let $a$ and $b$ be essential simple closed curves on $S$. Then $T_a^k=T_b^l$ if and only if $a=b$ and $k=l$ \cite[Section 3.3]{MR2850125}.
\end{enumerate}
\end{lemma}

\begin{lemma} \cite[Corollary 3.2]{MR2399650} \label{lem:autinn}
Let $G$ be a group, $\phi \in \Aut(G)$ and $g \in G$. Then $R(\phi \, \hat{g}) = R(\phi)$. In particular, $R(\hat{g})$ is the number of conjugacy classes in $G$.
\end{lemma}

\begin{proposition} \label{prop:order2}
Let  $S$ be a connected orientable infinite-type surface without boundary such that either $g \geq 2$ (can be infinite) or every compact subsurface of $S$ is displaceable. Then there exists an essential simple closed curve $c$ and an orientation reversing homeomorphism $f$ of $S$ of order two such that $\{c, f(c)\}$ forms a multicurve.
\end{proposition}

\begin{proof}
We have the following two cases:
\par
Case 1: Suppose that $g \geq 2$ (including infinite genus). We assume that the surface $S$ is symmetric about the $xy$-plane when embedded in $\mathbb{R}^3$. 
Then the orientation reversing homeomorphism $\m$ and the curve $c$ shown in Figure~\ref{fig:np1} have the property that $\{c, \m(c)\}$ forms a multicurve.
	\begin{figure}[H]
		\labellist
		\tiny
		\pinlabel $c$ at 300, 80
		\pinlabel $x$ at 380, 65
		\pinlabel $y$ at 195, 140
		\endlabellist
		\centering
		\includegraphics[scale=0.57]{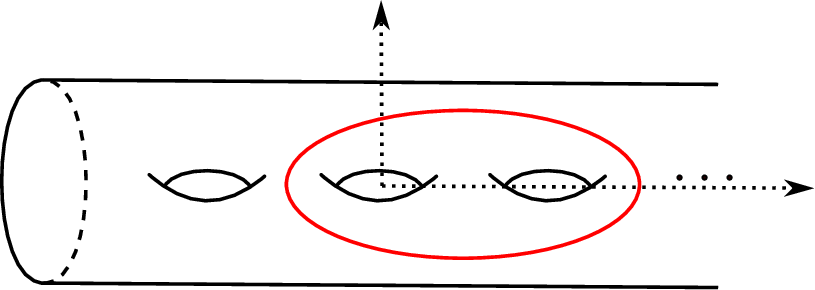}
		\caption{The curve $c$ such that $\{c, \m(c)\}$ forms a multicurve.}
		\label{fig:np1}
	\end{figure}
\par
Case 2: Suppose that every compact subsurface of $S$ is displaceable. In this case,  the genus of $S$ is either zero or infinite. The latter possibility is already considered in Case 1. So, we assume that the genus is zero, that is, $S$ is planar.  In view of Theorem \ref{compact subsurface is displaceable}, we have three possibilities. If $S$ has a unique maximal end, then by the proof of Proposition~\ref{prop:uniquemax}, there exist curves $\{c, c^{\prime}, b\}$ which bound a pair of pants and $\{c,c'\} \subset \Sigma_b$ as shown in Figure \ref{fig:unique}. Now, we take $f$ to be the orientation reversing homeomorphism of $S$ which maps a point on $S$ to its mirror image about the $xz$-plane. Then the collection $\{ c, f(c)=c' \}$ forms a multicurve.
	\begin{figure}[H]
		\labellist
		\tiny
		\pinlabel $x$ at 365, 120
		\pinlabel $c'$ at 67, 170
		\pinlabel $c$ at 70, 65
		\pinlabel $b$ at 185, 140
		\pinlabel $y$ at 125, 290
		\pinlabel $\Omega_b$ at 235, 180
		\pinlabel $\Sigma_{c}$ at 50, 0
		\pinlabel $\Sigma_{c'}$ at 50, 230
		\endlabellist
		\centering
		\includegraphics[scale=0.57]{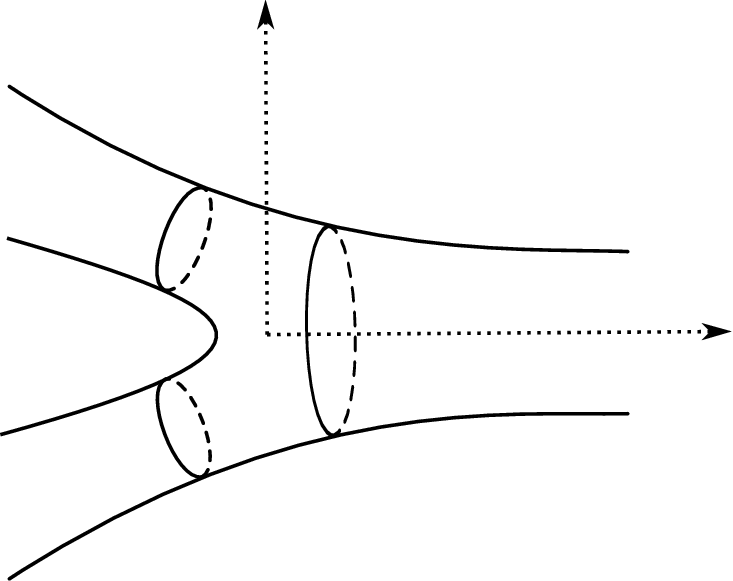}
		\caption{The curve $c$ such that $\{c, c'\}$ forms a multicurve.}
		\label{fig:unique}
	\end{figure}
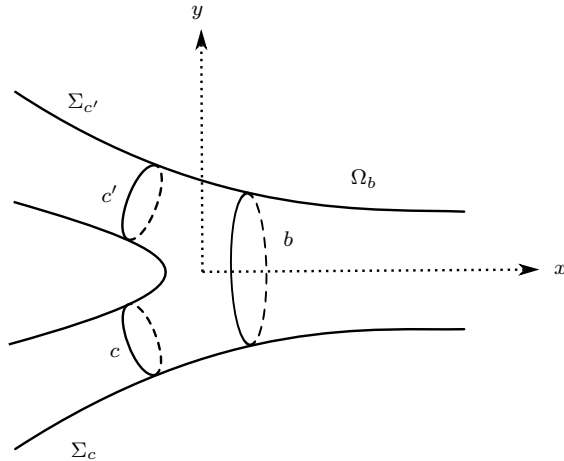
If $S$ has two maximal ends, then by \cite[Lemma 5.4]{MR4492497}, we can realise our surface $S$ such that there exists an orientation reversing homeomorphism $f$ of $S$ of order two, which permutes the two maximal ends. If $c$ is a curve in a neighbourhood of one of the maximal ends, then $\{ c, f(c) \}$ forms a multicurve.
\par 

If $S$ is such that the space $\mathcal{M}(S)$ of maximal ends is homeomorphic to the cantor space, then we take $f$ to be the orientation reversing homeomorphism of $S$ which maps a point of $S$ to its mirror image about the $xz$-plane (see Figure~\ref{fig:curve}). In this case,  $\{b_1, f(b_1)=b_5 \}$ forms a multicurve.
\end{proof}

\begin{proposition}\label{prop:inftyconj}
Let $S$ be a connected orientable infinite-type surface without boundary. Let $f$ be an orientation reversing homeomorphism of $S$ of order two such that there exists an essential simple closed curve $c$ for which $\{c, f(c)\}$ is a multicurve. Then the right coset $\Mod(S)f$ contains infinitely many conjugacy classes of the group $\Mod^{\pm}(S)$.
\end{proposition}

\begin{proof}
We claim that, for each $n \in \mathbb N$, the element  $T_c^n f$ represents a distinct conjugacy classes in $\Mod^{\pm}(S)$. We set $d:=f(c)$. Suppose that $T_c^n f$ is conjugate to $T_c^m f$ in $\Mod^{\pm}(S)$ for some $m,n \ge 1$. This implies that $(T_c^n f)^2 = T_c^n T_{d}^{-n}$ is conjugate to $(T_c^m f)^2= T_c^m T_{d}^{-m}$ in $\Mod^{\pm}(S)$. Thus, there exists $g \in \Mod^{\pm}(S)$ such that
$$T_c^{n}T_d^{-n}  = gT_c^{m} T_d^{-m}g^{-1}	 = T_{g(c)}^{\epsilon m} T_{g(d)}^{- \epsilon m},$$
	where $$\epsilon = \begin{cases}
	1 & \text{ if } g \text{ is orientation preserving,}\\
	-1 & \text{ if } g \text{ is orientation reversing.}
	\end{cases}$$
	By Lemma~\ref{prop:multicurve}, this is possible only when $m = n$.
\end{proof}

\begin{theorem}\label{theorem r-infinity property}
Let  $S$ be a connected orientable infinite-type surface without boundary. Then $\Mod(S)$ has the $R_\infty$-property.
\end{theorem}

\begin{proof}
It follows from Lemma~\ref{lem:twistprop1} and Lemma \ref{prop:multicurve} that if $c$ is an essential simple closed curve on $S$, then $T_c^m$ is conjugate to $T_c^n$ in $\Mod(S)$ if and only if $m = n$. Thus, $\Mod(S)$ has infinitely many conjugacy classes. If $S$ admits a non-displaceable compact subsurface of genus $0$ or $1$, then the theorem follows from \cite[Theorem 3.4]{MR4525669}. So, we assume that either every compact subsurface of $S$ is displaceable or $S$ has genus at least two.   Since $\Out(\Mod(S)) \cong \mathbb{Z}_2$, in view of Lemma~\ref{lem:autinn}, it is enough to show that $R(\hat{f})$ is infinite for an orientation reversing homeomorphism $f$ of $S$ of order two. By Proposition~\ref{prop:order2}, we can take $f$ to be such that there exists an essential simple closed curve $c$ for which $\{c, f(c)\}$ forms a multicurve. By Proposition \ref{prop:inftyconj}, $\Mod(S)f$ contains infinitely many conjugacy classes of the group $\Mod^{\pm}(S)$. Finally, Lemma \ref{lemma:twistedconj1} implies that  $R(\hat{f})$ is infinite, which completes the proof.
\end{proof}

\section{Further directions}
We conclude with the following questions that arose naturally during our investigation.
\begin{enumerate}
\item Our findings indicate that when $S$ is an infinite-type surface without boundary, the mapping class group $\Mod(S)$ either exhibits the $\phi$-twisted Rokhlin property for all automorphisms $\phi$, or it lacks the $\phi$-twisted Rokhlin property for each automorphism $\phi$. We do not know an example of a topological group $G$ which admits a pair of automorphisms $\phi,\psi$  such that  $G$  has the $\phi$-twisted Rokhlin property, but not the $\psi$-twisted Rokhlin property? In general, it would be interesting to find such examples in topological groups that admit `sufficiently large' outer automorphim groups.
\item Suppose that $H$ is a (closed) subgroup of a topological group $G$ such that each automorphism of $H$ is the restriction of an inner automorphism of $G$. Is it true that $H$ has the Rokhlin property if and only if it has the twisted Rokhlin property?
\item Which other interesting topological groups admit the twisted Rokhlin property?
\end{enumerate}
\medskip

\begin{ack}
Pravin Kumar is supported by the PMRF fellowship at IISER Mohali. Mahender Singh is supported by the Swarna Jayanti Fellowship grants DST/SJF/MSA-02/2018-19 and SB/SJF/2019-20/04.
Apeksha Sanghi acknowledges post-doctoral fellowship from the grant SB/SJF/2019-20/04. 
\end{ack}

\section{Declaration}
The authors declare that they have no conflicts of interest and that there is no data associated with this paper.

\bibliographystyle{plain}

\end{document}